\documentclass[12pt,draft]{article}
\usepackage{amsmath,amscd}
\usepackage{amssymb,latexsym,amsthm}
\usepackage{color}
\usepackage{amssymb}
\usepackage{color}
\usepackage{amsmath,amsthm,amscd}
\usepackage[latin1]{inputenc}
\usepackage{lscape}
\usepackage{fancyhdr}
\usepackage{amsfonts}
\usepackage{pb-diagram}

\numberwithin{equation}{section}
\newtheorem{theorem}{Theorem}[section]

\newtheorem{proposition}[theorem]{Proposition}

\newtheorem{corollary}[theorem]{Corollary}

\theoremstyle{definition}
\newtheorem{example}[theorem]{Example}
\newtheorem{definition}[theorem]{Definition}

\newtheorem{remark}[theorem]{Remark}

\newcommand{\cR}{\mbox{${\cal R}$}}

\title{\textbf{A generalized Koszul property\\ for skew PBW extensions}}
\author{\\  H\'ector Su\'arez\footnote{Escuela de Matem\'aticas y Estad\'istica,
Universidad Pedag\'ogica y Tecnol\'ogica de Colombia, Tunja. } \footnote{Departamento de Matem\'aticas,  Universidad Nacional de Colombia, sede Bogot\'a. } \\Armando Reyes \footnote{Departamento de Matem\'aticas,
Universidad Nacional de Colombia - sede Bogot\'a.}}
\date{}
\begin{document}
\maketitle
\begin{abstract}
\noindent
Let $R$ be a commutative algebra. In this paper  we show that constant skew PBW extensions of a generalized Koszul algebra $R$ are also generalized Koszul. Let $A$ be a semi-commutative skew PBW extension of $R$ such that $A$ is $R$-augmented. We show also that if $R$ is augmented Koszul  then $A$ is $R$-augmented Koszul.
\bigskip


\bigskip

\bigskip

\end{abstract}

\section{Introduction}

Koszul algebras were first introduced by Priddy in  \cite{Priddy1970} under the name of homogeneous Koszul algebras and have since revealed important applications in algebraic geometry, Lie theory,
quantum groups, algebraic topology and combinatorics. The  structure and  history of Koszul algebras are  detailed
in \cite{Polishchuk}. Backelin and Fr\"oberg in \cite{Backelin} presented several equivalent definitions of Koszul algebras. Later they emerged some general notions of Koszul algebras (or Koszul rings). For example in \cite{BeilinsonGinzburgSoerge1996}, a \textit{Koszul ring} is a positively graded ring $A =\bigoplus_{j\geq 0}A_j,$ such that
 $A_0$ is semisimple and $A_0$ considered as a graded left $A$-module admits a graded projective resolution \[
\dotsb \to P^{n} \to \dotsb \to P^{1} \to P^{0} \twoheadrightarrow A_0,
\]
such that $P^i$ is generated by its degree $i$ component, i.e., $P^i = AP^i_i$. Many interesting algebras with properties
similar to a Koszul algebra do not satisfy the assumption that $A_0$ is semisimple. There do already exist several generalized Koszul theories where the degree 0 part
$A_0$ of a graded algebra $A$ is not required to be semisimple, see \cite{GreenReitenSolberg2002}, \cite{Li5}, and \cite{Woodcock1998}. Each
Koszul ring $A$ defined by Woodcock in \cite{Woodcock1998}  is supposed to satisfy that $A$ is both
a left projective $A_0$-module and a right projective $A_0$-module. This requirement is too strong. In \cite{Li5}, it was defined the notion of generalized Koszul modules
and Koszul algebras in a similar way to the classical case. Phan in \cite{Phan1} and \cite{Phan2} defined Koszul algebras for augmented algebras and $R$-augmented algebras.\\

Skew $PBW$ extensions or $\sigma$-PBW extensions were defined in 2011 by Gallego and Lezama \cite{LezamaGallego}. Several properties of these rings have been studied (see for example \cite{LezamaGallego}, \cite{Reyes2013}, \cite{Reyes2014}, \cite{LezamaReyes},  \cite{LezamaAcostaReyes2015}, \cite{Reyes2015}, \cite{Venegas2015}, \cite{ReyesSuarez2016a}, and others). In \cite{SuarezLezamaReyes2016} we study the Koszul property (in the direction of Priddy)  for skew PBW extensions over fields.  Since the skew PBW extensions are rings, our interest in this paper is to study the generalized Koszul property for skew extensions, according to the definitions given by Li and Phan in \cite{Li5} and \cite{Phan1}, respectively.
\section{Graduation of skew PBW extensions}
In this section we  will see some cases for which a skew  PBW extension of a graded algebra $R$ is also graded. Throughout this paper $\mathbb{K}$ is a field. We first  recall the definition and some properties of skew PBW extension introduced  in  \cite{LezamaGallego}, \cite{LezamaReyes} and \cite{SuarezLezamaReyes2016}. Note that the following definition extends Ore extensions and PBW extensions.
\begin{definition}[\cite{LezamaGallego}, Definition 1]\label{def.skewpbwextensions}
Let $R$ and $A$ be rings. We say that $A$ is a \textit{skew PBW
extension of} $R$ (also called a $\sigma$-PBW extension
of $R$) if the following conditions hold:
\begin{enumerate}
\item[\rm (i)]$R\subseteq A$;
\item[\rm (ii)]there exist elements $x_1,\dots ,x_n\in A$ such that $A$ is a left free $R$-module, with basis the basic elements
\begin{center}
${\rm Mon}(A):= \{x^{\alpha}=x_1^{\alpha_1}\cdots
x_n^{\alpha_n}\mid \alpha=(\alpha_1,\dots ,\alpha_n)\in
\mathbb{N}^n\}$.
\end{center}
In this case it is said also that $A$ is a left polynomial ring over $R$ with respect
to $\{x_1,\cdots, x_n\}$ and Mon(A) is the set of standard monomials of $A$. Moreover,
$x^0_1\cdots x^0_n := 1 \in {\rm Mon}(A)$.

\item[\rm (iii)]For each $1\leq i\leq n$ and any $r\in R\ \backslash\ \{0\}$, there exists an element $c_{i,r}\in R\ \backslash\ \{0\}$ such that
\begin{equation}\label{sigmadefinicion1}
x_ir-c_{i,r}x_i\in R.
\end{equation}
\item[\rm (iv)]For any elements $1\leq i,j\leq n$, there exists $c_{i,j}\in R\ \backslash\ \{0\}$ such that
\begin{equation}\label{sigmadefinicion2}
x_jx_i-c_{i,j}x_ix_j\in R+Rx_1+\cdots +Rx_n.
\end{equation}
Under these conditions we will write $A:=\sigma(R)\langle
x_1,\dots,x_n\rangle$.
\end{enumerate}
\end{definition}

We can  build new skew PBW extensions  from a given skew PBW extension.
\begin{proposition}\label{inocente}
 If $A$ is a skew PBW extension of a commutative ring $K$, and $B$ is a commutative $K$-algebra, then $B\otimes_K A$ is a skew PBW extension of $B\otimes_KK$, i.e., a skew PBW extension of $B$.
\begin{proof}
We need verify the four conditions of the Definition \ref{def.skewpbwextensions}.
\begin{enumerate}
\item[\rm (i)] Since $K\subseteq A$, it is clear that $B\otimes_K K \subseteq B\otimes_K A$.
 \item[\rm (ii)] We know that $B\otimes_K A$ is a $B$-algebra under the product $b'(b\otimes a):=b'b\otimes a$, and that  $B\otimes_K A$ is $B$-free with the same rank that $A$ as $K$-module (\cite{Rowen2008}, Remark 18.27), and since $B\cong B\otimes_K K$, we have that $B\otimes_K A$ is a free $B\otimes_K K$-module. We  also note  that
\[
{\rm Mon}(B\otimes_K A) = \{(1\otimes x_1)^{\alpha_1} \dotsb (1\otimes x_n)^{\alpha_n}\mid \alpha_i\in \mathbb{N}, 1\le i\le n\},
\]
where $B\otimes_K A$ is $B\otimes_K K$-free with the same rank that $A$ as $K$-module.
\item[\rm (iii)] Let $r\otimes k'$ be a nonzero element of $B\otimes_K K$ (which we identify with  $rk'$ since $B\cong B\otimes_K K$), we can see that for all element $1\otimes x_i$ in the basis, there exist an element $1\otimes c_{i,r}\in B\otimes_K K$ such that
\[
(1\otimes x_i)(r\otimes k') - (1\otimes c_{i,r})(1\otimes x_i) \in B\otimes_K K,
\]
where $c_{i,r}\in K$ satisfies $x_ir-c_{i,r}x_i \in K$, being as $A$ is a skew PBW extension of $K$.
\item[\rm (iv)] The condition
\begin{multline*}
(1\otimes x_i)(1\otimes x_j) - (1\otimes c_{i,j})(1\otimes x_j)(1\otimes x_i)\\
 \in B\otimes_K K + (B\otimes_K K)(1\otimes x_1) + \dotsb + (B\otimes_K K)(1\otimes x_n),
\end{multline*}
it follows from the  Definition \ref{def.skewpbwextensions}, condition (iv).
\end{enumerate}
\end{proof}
\end{proposition}
\begin{example}
If $A=\sigma(K)\langle x_1,\dotsc, x_n\rangle$, and $K[y]:=K[y_1,\dotsc, y_m]$, where $y_1, y_2,\dotsc, y_m$ are new variables, then
\[
K[y]\otimes_K A\cong \sigma (K[y]\otimes_K K)\langle 1\otimes x_1,\dotsc, 1\otimes x_n\rangle\cong \sigma(K[y])\langle x_1,\dotsc, x_n\rangle.
\]
\end{example}
The notation $\sigma(R)\langle x_1,\dots,x_n\rangle$ and the name of the skew PBW extensions is due to the following proposition.

\begin{proposition}[\cite{LezamaGallego}, Proposition 3]\label{sigmadefinition}
Let $A$ be a skew PBW extension of $R$. For each $1\leq i\leq
n$, there exist an injective endomorphism $\sigma_i:R\rightarrow
R$ and a $\sigma_i$-derivation $\delta_i:R\rightarrow R$ such that
$x_ir=\sigma_i(r)x_i+\delta_i(r)$, for every  $r \in R$.
\end{proposition}
There are some important sub-classes of skew PBW extensions (see for example \cite{LezamaGallego}, \cite{LezamaAcostaReyes2015} and \cite{SuarezLezamaReyes2016}).
\begin{definition}\label{sigmapbwderivationtype}
Let $A$ be a skew PBW extension of $R$, $\Sigma:=\{\sigma_1,\dotsc, \sigma_n\}$ and $\Delta:=\{\delta_1,\dotsc, \delta_n\}$, where $\sigma_i$ and $\delta_i$ ($1\leq i\leq n$) are as in  Proposition \ref{sigmadefinition}.
\begin{enumerate}
\item[\rm (a)]\label{def.constant} Any element $r$ of $R$ such that $\sigma_i(r)=r$ and $\delta_i(r)=0$, for all $1\leq i\leq n$, will be called a {\em constant}. $A$ is called \emph{constant} if every element of $R$ is constant.
\item[\rm (b)]  $A$ is called {\em pre-commutative} if the conditions  {\rm(}iv{\rm)} in Definition
\ref{def.skewpbwextensions} are replaced by:\\
For any $1\leq i,j\leq n$, there exists $c_{i,j}\in R\ \backslash\ \{0\}$ such that
$x_jx_i-c_{i,j}x_ix_j\in Rx_1+\cdots +Rx_n$.
\item[\rm (c)]\label{def.quasicom}  $A$ is called \textit{quasi-commutative} if the conditions
{\rm(}iii{\rm)} and {\rm(}iv{\rm)} in Definition
\ref{def.skewpbwextensions} are replaced by \begin{enumerate}
\item[\rm (iii')] for each $1\leq i\leq n$ and all $r\in R\ \backslash\ \{0\}$, there exists $c_{i,r}\in R\ \backslash\ \{0\}$ such that
$x_ir=c_{i,r}x_i$;
\item[\rm (iv')]for any $1\leq i,j\leq n$, there exists $c_{i,j}\in R\ \backslash\ \{0\}$ such that $x_jx_i=c_{i,j}x_ix_j$.
\end{enumerate}
\item[\rm (d)] $A$ is called \textit{bijective} if $\sigma_i$ is bijective for each $\sigma_i\in \Sigma$, and $c_{i,j}$ is invertible,  for any $1\leq
i<j\leq n$.
\item[\rm (e)]  If $\sigma_i={\rm id}_R$, for every $\sigma_i\in \Sigma$, we say that $A$ is a skew PBW extension of {\em derivation type}.
\item[\rm (f)]  If $\delta_i = 0$, for every $\delta_i\in \Delta$, we say that $A$ is a skew PBW extension of {\em endomorphism type}.
\item[\rm (g)]  $A$ is called {\em semi-commutative} if $A$ is quasi-commutative and constant.
\end{enumerate}
\end{definition}
Specific examples of the above classes of skew PBW extensions can be consulted in \cite{SuarezLezamaReyes2016}.\\

Quasi-commutative skew PBW extensions can be obtain from a given skew PBW extension of a ring $R$.

\begin{proposition}[\cite{LezamaReyes},  Proposition 2.1]\label{remarkAsigma}  Let $A$ be a skew PBW extension of $R$. Then, there exists a quasicommutative
skew PBW extension $A^{\sigma}$ of $R$ in $n$ variables $z_1,\dotsc, z_n$ defined by the relations $z_ir=c_{i,r}z_i$, $z_jz_i=c_{i,j}z_iz_j$, for $1\le i\le n$, where $c_{i,r}, c_{i,j}$ are the same constants that define $A$. Moreover, if $A$ is bijective then $A^{\sigma}$ is also bijective.
\end{proposition}

The next proposition  computes the graduation of a general skew
PBW extension of a ring $R$.

\begin{proposition}[\cite{LezamaReyes}, Theorem 2.2]\label{teo.Gr(A)}
Let $A$ be an arbitrary skew PBW extension of $R$. Then, $A$ is
a filtered ring with increasing filtration given by
\begin{equation}\label{eq1.3.1a}
F_m(A):=\begin{cases} R & {\rm if}\ \ m=0\\ \{f\in A\mid {\rm deg}(f)\le m\} & {\rm if}\ \ m\ge 1
\end{cases}
\end{equation}
and the corresponding graded ring $Gr(A)$ is isomorphic to $A^{\sigma}$, where $A^{\sigma}$ is as in Proposition \ref{teo.Gr(A)}.
\end{proposition}

In general, skew PBW extensions do not admit nontrivial graduation. In this section, we will show that if $R$ is a  graded commutative ring, then every constant skew PBW extension of $R$ is also a graded ring. Graduation extends from  $R$ to $A$.\\


For the following proposition we assume that $K$ is a commutative ring and $G$ is a group. A ring $B$ is called a graded $K$-algebra, if $B$ is a $G$-graded ring, and $B$ is a $K$-algebra with $a\cdot 1 = 1\cdot a\in B_{e}$, for all $a\in K$, where $e$ denotes the identity  element of the group.

\begin{proposition}\label{NastaVOEjemplo1.3.11}
If $A$ is a constant skew PBW extension of a commutative $G$-graded ring $R$, then $A$ is also a $G$-graded ring.
\begin{proof}
 From \cite{NastaVO}, Example 1.3.11, we have that if $B$ is a $G$-graded $K$-algebra, and $S$ is a $K$-algebra, then the tensor product $B\otimes_K S$ is a $K$- algebra with the product $(r\otimes s)(r'\otimes s'):=rr'\otimes ss'$, for $r, r' \in B$, and, $s, s'\in S$. Also,

\begin{equation}\label{grad.prod}
B\otimes_K S = \sum_{g\in G} B_g\otimes_K S,
\end{equation}

thus, we obtain a $G$-graduation of $B\otimes_K S$. If $A$ is a constant skew PBW extension of a  $G$-graded commutative ring $R$, the $R$-isomorphism $R\otimes_R A\cong A$ allows to consider  $A$ as a $G$-graded $R$-algebra.
\end{proof}
\end{proposition}

Let $A$ be a skew PBW extension of $R$. Each element $f\in A\ \backslash\ \{0\}$ has a unique representation as $f=c_1X_1+\cdots+c_tX_t$, with $c_i\in R\ \backslash\ \{0\}$ and $X_i= x^{\alpha}= x_1^{\alpha_1}\cdots
x_n^{\alpha_n}\in {\rm Mon}(A)$ for $1\leq i\leq t$ (Remark 2, \cite{LezamaGallego}). For $\alpha=(\alpha_1,\dots,\alpha_n)\in \mathbb{N}^n$,
$\sigma^{\alpha}:=\sigma_1^{\alpha_1}\cdots \sigma_n^{\alpha_n}$,
$|\alpha|:=\alpha_1+\cdots+\alpha_n$.

\section{Generalized Koszul property}

Let $A=A_0\bigoplus A_1 \bigoplus \cdots$ be a graded algebra. Denote $J =\bigoplus_{i\geq 0}A_i$, which is a two-sided ideal of $A$. We say that $A$ is \emph{locally finite}  if $\dim_{\mathbb{K}} A_i<\infty$, for all  $i\geq 0$. $A$ is \emph{ge\-ne\-ra\-ted  in degrees 0 and 1} if  $A_1\cdot A_i = A_{i+1}$, for all $i\geq 0$. An $A$-module $M = \bigoplus_{i\in \mathbb{Z}}M_i$ such that $M_i=0$, for $i \ll 0$ is called \emph{graded} if  $A_i\cdot M_j\subseteq M_{i+j}$.  $M$ is \emph{locally finite} if $\dim_{\mathbb{K}}M_i<\infty$, for all $i\in \mathbb{Z}$. In this paper all graded modules are supposed to be locally finite.  A module $M$ such that there exists $l$ satisfying $M = M_l$ is said to be \emph{concentrated  in degree $l$}. We identify $A_0$ with the quotient module $A/J$ and view it as a graded $A$-module concentrated in degree 0. We say that $M$ is generated in degree $s$ if $M = A\cdot M_s$. If $M$ is generated in degree $s$, then $\bigoplus_{i\geq s+1}M_i$ is a graded submodule of $M$, and $M_s \cong M/\bigoplus_{i\geq s+1}M_i$ as vector spaces. We then view $M_s$ as an $A$-module by identifying it with this quotient module. $M$ is generated in degree $s$ if and only if $JM \cong\bigoplus_{i\geq s+1}M_i$, which is equivalent to $J^lM \cong\bigoplus_{i\geq s+l}M_i$, for all $l \geq 1$ (see  \cite{Li5}).\\

Let $A=A_0\bigoplus A_1 \bigoplus \dotsb$ be a locally finite graded algebra. A PBW
\emph{deformation} of $A$ is a filtered algebra $U$ with an ascending filtration $0 = F_{-1}(U) \subseteq F_0(U) \subseteq F_1(U) \subseteq\cdots$ such that the associated graded algebra $Gr(U)$ is isomorphic to $A$.

\begin{proposition}
Let $A$ be a skew PBW extension of $\mathbb{K}$. Then $A$ is a PBW deformation of $A^{\sigma}$.
\end{proposition}
\begin{proof}
$A$ is a filtered algebra with an ascending filtration $0 = F_{-1}(A) \subseteq F_0(A) \subseteq F_1(A) \subseteq \cdots$ as in  (\ref{eq1.3.1a}). By Theorem \ref{teo.Gr(A)} the associated graded algebra $Gr(A)$ is isomorphic to $A^{\sigma}$.
\end{proof}

\begin{definition}[\cite{Li5}, Definition 2.3]\label{def.koszul} A graded $A$-module $M$ is called a \emph{generalized Koszul module} if it has a (minimal) linear projective resolution
\begin{equation}\label{res.Koszul}
\to P^n\to  P^{n-1}\to\cdots\to P^0\to M\to 0,
\end{equation}
such that $P^i$ is generated in degree $i$ for all $i\geq 0$. The graded algebra $A$ is called a \emph{generalized Koszul algebra} if $A_0$ viewed as an $A$-module is generalized Koszul.
\end{definition}

A graded algebra $A=\bigoplus_{p\geq 0}A_p$ is called \emph{connected} if $A_0=\mathbb{K}$. Let $L:=\mathbb{K}\langle x_1,\dots, x_n\rangle$ be the free associative algebra (tensor algebra) in $n$ generators $x_1,\dots, x_n$. Note that $L$ is positively graded with graduation given by $L:=\bigoplus_{j\geq 0}L_j$ where $L_0= \mathbb{K}$ and $L_j$ spanned by all words of length $j$ in the alphabet $\{x_1, \dots, x_n\}$, for $j>0$. Let  $P$ be a subspace of $F_2(L):=\mathbb{K}\bigoplus L_1\bigoplus L_2$, the algebra $L/\langle P\rangle$ is called (nonhomogeneous) \emph{quadratic algebra}. $L/\langle P\rangle$ is called \emph{homogeneous quadratic algebra} if $P$ is a subspace of $L_2$, where $\langle P\rangle$ the two-sided ideal of $L$ generated by $P$. Let $I\subseteq \sum_{j\geq 2} L_{j}$ be a finitely-generated homogeneous ideal of $\mathbb{K}\langle x_1,\dots, x_n\rangle$ and let $R = \mathbb{K}\langle x_1,\dots, x_n\rangle/I$, which is a connected-graded $\mathbb{K}$-algebra generated in degree 1. Suppose $\sigma : R \to R$ is a graded algebra automorphism and $\delta : R(-1) \to R$ is a graded $\sigma$-derivation (i.e. a degree +1 graded $\sigma$-derivation $\delta$ of $R$). Let  $A := R[x; \sigma,\delta]$  be the associated \emph{graded Ore extension} of $R$; that is, $A = \bigoplus_{n\geq 0} Rx^n$ as an $R$-module, and for $r\in R$, $xr = \sigma(r)x + \delta(r)$. We consider $x$ to have degree 1 in $A$, and under this grading $A$ is a connected graded algebra generated in degree 1 (see \cite{Cassidy2008} and \cite{Phan}). Note that $A$ is quadratic if and only if $R$ is quadratic, and it is well known that $A$ is Koszul (in the classical sense) if and  only if $R$ is Koszul (see for example \cite{Phan},  Corollary 1.3). In the following theorem we show a similar result for skew $PBW$ extensions.

\begin{theorem}\label{result.princ}
If $R$ is a commutative generalized Koszul $\mathbb{K}$-algebra and  $A$ is a constant skew PBW extension of $R$, then $A$ is a generalized Koszul algebra.
\end{theorem}
\begin{proof}
Let $R$ be a commutative generalized Koszul algebra, then: $R=\bigoplus_{j\geq 0}R_j$,
and $R_0$  as a graded right  $R$-module admits
a graded projective resolution
\begin{equation}\label{resol.R0}
\dotsb \xrightarrow{f_{i+1}} P^{i} \xrightarrow{f_i} \dotsb \xrightarrow{f_2} P^{1} \xrightarrow{f_1} P^{0} \xrightarrow{f_0} R_0 \to 0,
\end{equation}
such that $P^i$ is generated by its degree $i$ component, i.e., $P^i = P^i_i\cdot R$.\\

Let $A=\sigma(R)\langle x_1,\dots, x_n\rangle$ be a skew $PBW$ extension of $R$. Since $A$ is constant, then $A$ is a $R$-algebra and by Proposition \ref{NastaVOEjemplo1.3.11} $A$ is a graded algebra with graduation $A =  \bigoplus_{j\geq 0}A_j$ where
\begin{equation}\label{comp.Aj}
A_j= R_j\otimes_R A.
\end{equation}

 Now we note  that
\begin{equation}\label{resol.A0}
\dotsb \xrightarrow{f_{i+1}\otimes\iota_A}  P^{i}\otimes_R A  \xrightarrow{f_{i}\otimes\iota_A} \dotsb \xrightarrow{f_{2}\otimes\iota_A} P^{1}\otimes_R A \xrightarrow{f_{1}\otimes\iota_A} P^{0}\otimes_R A \xrightarrow{f_{0}\otimes\iota_A} R_0\otimes_R A \to 0
\end{equation}
is a graded projective resolution of $A_0$ as a graded right $A$-module. In effect: by Equation \ref{comp.Aj}, $R_0\otimes_R A=A_0$; if  $f_i: P^i \to P^{i-1}$ is a graded $R$-homomorphism then $f_i(P^i_j)\subseteq P^{i-1}_j$ and therefore $f_i\otimes i_A(P_j^{i}\otimes_R A)\subseteq P^{i-1}_j\otimes_R A$ ($\iota_A$ denotes the identical homomorphism), so  $f_i\otimes \iota_A: P^{i}\otimes_R A \to P^{i-1}\otimes_RA$   is a graded $A$-homomorphism  and (\ref{resol.A0})  is exact;
  since $P^i$ is a right graded $R$-module and $A$ is a left graded $R$-algebra, then by proof of Proposition \ref{NastaVOEjemplo1.3.11}, Equation (\ref{grad.prod}), $P^{i}\otimes_R A=\bigoplus_j(P_j^i\otimes_R A)$ is a  graded right $A$-module, for all $i\geq 0$; and
     $P^i\otimes_R A$ is a right projective $A$-module, for all $i\geq 0$.\\

We see now that for all  $i\geq 0$, $P^i\otimes_RA$ is generated by its degree $i$ component, i.e., $P^i\otimes_RA = (P^i_i\otimes_RA)\cdot A$. In effect: since  $P^i = P^i_i\cdot R$, then for $h\in  P^i$ $h =h_{i_1}r_1 + \cdots + h_{i_n}r_n$, $h_{i_s}\in P_i^i$, $r_s\in R$ ($1\leq s\leq n$, $n\geq 1$); so for $h\otimes a\in P^i\otimes_RA$ ($h\in P^i$, $a\in A$),
\begin{multline*}
h\otimes a= (h_{i_1}r_1 + \cdots + h_{i_n}r_n)\otimes a = h_{i_1}r_1\otimes a + \cdots + h_{i_n}r_n\otimes a\\
= (h_{i_1}\otimes r_1 + \cdots + h_{i_n}\otimes r_n)a\in (P^i_i\otimes_RA)\cdot A.
\end{multline*}
\end{proof}

We note that if $ A =\bigoplus_{j\geq 0}A_j$ is generalized Koszul and $A_0$ is semisimple then $A$ is  classical Koszul, i.e., in the sense of definition given by Beilinson, Ginzburg and Soergel in \cite{BeilinsonGinzburgSoerge1996}. Moreover if $A_0=\mathbb{K}$ is a field then $A$ is homogeneous Koszul, i.e., according to the definition given by Priddy in \cite{Priddy1970}.

\begin{example}\label{ex.diffus alg2} Diffusion algebra. The  diffusion algebra $\mathcal{D}$ is generated by $\{D_i, x_i \mid 1 \leq i \leq n\}$ over $\mathbb{K}$ with relations $x_ix_j = x_jx_i$, $x_iD_j = D_jx_i$, $1 \leq i, j \leq n$; \quad $c_{ij}D_iD_j - c_{ji}D_jD_i = x_jD_i - x_iD_j$, $i < j$, $c_{ij}$, $c_{ji}\in  \mathbb{K}^*$. Thus, $A \cong \sigma(\mathbb{K}[x_1, \dots, x_n])\langle D_1,\dots, D_n\rangle$ is a constant non quasi-commutative
skew $PBW$ extension of a commutative  Koszul $\mathbb{K}$-algebra  $R=\mathbb{K}[x_1, \dots, x_n]$. Then by Theorem \ref{result.princ}, the diffusion algebra  $\mathcal{D}$ is generalized Koszul.
\end{example}

For each $n\geq 0$, \emph{the syzygy functors} $\Omega^i$ is $\Omega^n(M):=ker(P_{n-1}\to P_{n-2})$, where $\cdots \to P_{2} \rightarrow P_{1} \rightarrow P_{0} \rightarrow M\to 0 $ is a minimal projective resolution of $M$. Every graded free $A$-module can be written as  $A\otimes_{\mathbb{K}}V$ where $V$ is a graded vector space and $A\otimes_{\mathbb{K}}V$ is given the tensor product grading.

\begin{proposition}[\cite{Li5}]\label{prop.Koszul syzy}
$M$ is a generalized Koszul $A$-module if and only if $M$ is generated in degree 0 and each syzygy $\Omega^i(M)$ is generated in degree $i$, for every $i \geq 1$. Moreover, from the  projective resolution (\ref{res.Koszul}), we have that $M_0\cong P_0^0$ and $\Omega^i(M)_i\cong P_i^i$ are projective $A_0$-modules, for all $i \geq 1$.
\end{proposition}

\begin{remark}
From the Proposition \ref{prop.Koszul syzy} we have that skew PBW extensions are not generalized Koszul  $R$-modules. In fact, if $A=\sigma(R)\langle x_1,\dots x_n\rangle$ is a generalized Koszul $R$-module (with $x_1, x_2,\dots, x_n$ in degree 1), then $A$ is generated in degree 0, i.e., $A=R\cdot A_0$, which is a contradiction.
\end{remark}

Let $A =\bigoplus_{i\geq 0}A_i$ be a locally finite graded algebra generated in degrees 0 and 1. The \emph{graded radical} grad($A$) is the intersection of all maximal proper graded submodules of $A$. Let $\textbf{r}$ be the radical of $A_0$, and $\cR := \langle \textbf{r}\rangle$ be the two-sided ideal of $A$ generated by $\textbf{r}$. We then define the quotient graded algebra $\bar{A}:=A/\cR =\bigoplus_{i\geq 0}A_i/\cR_i$. $\bar{A}$ is a locally finite graded algebra for which the grading is induced from that of $A$, and $\cR_s=\sum_{i=0}^sA_i\textbf{r}A_{s-i}$. Note that $\bar{A}_0=A_0/\textbf{r}$ is a semisimple algebra. Given an arbitrary graded A-module $M=\bigoplus_{i\in \mathbb{Z}}M_i$ such that $M_i=0$, for $i\ll 0$, define $\bar{M}=M/\cR M=\bigoplus_{i\in \mathbb{Z}}M_i/(\cR M)_i$. Then $\bar{M}$ is a graded $\bar{A}$-module and $\bar{M}\cong \bar{A}\otimes_AM$ (see \cite{Li5}). Let $A=\sigma(R)\langle x_1,\dots, x_n\rangle=\bigoplus_{i\geq 0}A_i$ be a graded skew PBW extension of $R=\bigoplus_{j\geq 0}R_j$,  $\Sigma:=\{\sigma_1,\dotsc, \sigma_n\}$ and $\Delta:=\{\delta_1,\dotsc, \delta_n\}$, where $\sigma_i$ and $\delta_i$ ($1\leq i\leq n$) are as in the Proposition \ref{sigmadefinition}. If $I$ is an ideal of $R$, $I$ is called $\Sigma$-\emph{invariant} if $\sigma_i(I)\subseteq I$  and $I$ is called $\Delta$-\emph{invariant} if $\delta_i(I)\subseteq I$, for every $1\leq i\leq n$ (see \cite{LezamaAcostaReyes2015}, Definition 2.1).  Note that $\textbf{r}$ is $\Sigma$-invariant but it is not $\Delta$-invariant because $\delta_i(\textbf{r})\subseteq A_1$ for each $\delta_i\in \Delta$, and therefore $\textbf{r}A_1\subseteq A_1$ but $\delta_i(\textbf{r}) A_1\subseteq A_2$. So, in general, $A_1\textbf{r} \neq\textbf{ r}A_1$. This leads us to propose the following.\\

\begin{proposition}\label{prop.radA1}
 If $A$ is a quasi-commutative bijective skew PBW extension of the $\mathbb{K}$-algebra $R$,  then $A_1\textbf{r} = \textbf{ r}A_1$.
\end{proposition}

\begin{proof} According to Proposition \ref{remarkAsigma},  $A\cong A^{\sigma}$. So, by Proposition \ref{teo.Gr(A)},  $A$ is a graded skew PBW extension with $R_0=A_0=R$, so $\textbf{r}$ is the radical of $R$. As $\sigma_i$ is bijective, for each $1\leq i\leq n$, then $\textbf{m}$ is maximal if and only $\sigma_i(\textbf{m})$ is maximal, so  $\sigma_i(\textbf{r})=\textbf{r}$. Note that $A_1$ is the space  generated by $\{r_1x_1,\dots, r_nx_n\}$ with $r_1,\dots r_n\in R_0=R=A_0$.
$\subseteq)$: Without loss of generality, suppose that $x_ia_0\in A_1\textbf{r}$ with $a_0$ in $\textbf{r}$, then $x_ia_0=\sigma_i(a_0)x_1\in\textbf{r} A_1$.\\
$\supseteq)$: Let $a_0x_i\in \textbf{ r}A_1$ with $a_0$ in $\textbf{r}$, then for $a_0$ there is $b_0\in \textbf{r}$ such that $\sigma_i(b_0)=a_0$ and for $b_0$, $a_0x_i= \sigma_i(b_0)x_i=x_ib_0\in A_1\textbf{r}$.
\end{proof}
According to the Proposition \ref{prop.radA1}, for a  quasi-commutative bijective skew PBW extension $A$ of the $\mathbb{K}$-algebra $R$, with $R$ concentrate in degree 0, we have that $\cR_s=\textbf{r}A_s$ and therefore $\bar{A}=\bigoplus_{i\geq 0}A_i/\textbf{r}A_i$ (see \cite{Li5}).

Note that if $A$ is a skew PBW extension as in Proposition
Li in  \cite{Li5}, shows the relationship between classical Koszul algebras and generalized Koszul algebras. Classical Koszul algebras are understood in the sense of the Definition 1.2.1 of  \cite{BeilinsonGinzburgSoerge1996}, i,e.,  \textit{Koszul algebra} is a positively graded $\mathbb{K}$-algebra $A =\bigoplus_{i\geq 0}A_i$ such that
 $A_0$ is semisimple and $A_0$ considered as a graded left $A$-module admits a graded projective resolution \[
\dotsb \to P^{n} \to \dotsb \to P^{1} \to P^{0} \twoheadrightarrow A_0,
\]
such that $P^i$ is generated by its degree $i$ component, i.e., $P^i = AP^i_i$.

\begin{theorem}  Let $A=\sigma(R)\langle x_1,\dots, x_n\rangle$ be a quasi-commutative bijective skew PBW extension of the locally finite $\mathbb{K}$-algebra $R$. Then $A$ is  a generalized Koszul algebra if and only if $\bar{A}=\bigoplus_{i\geq 0}A_i/\textbf{r}A_i$ is a classical Koszul algebra.
\end{theorem}
\begin{proof}
Note that $A$ is a locally finite  graded algebra. Indeed, $A$ is a graded algebra  and $\dim_\mathbb{K}A_0 = \dim_\mathbb{K}R_0=\dim_\mathbb{K}R$, $\dim_\mathbb{K}A_1 = \dim_\mathbb{K}R_1 + n=n$; let $\mathcal{B}$ be a (finite) base of $R=R_0$, then for a fixed $p\geq 2$ the set $\{rx^{\alpha} \mid |\alpha|= p,\  r\in B \text{  and } x^{\alpha}\in {\rm Mon}(A)\}$ is a finite base  for $A_p$. Now, as $A$ is  free as left $R$-module then $A$ is  a projective  $R_0$-module. By Proposition \ref{prop.radA1}, $A_1\textbf{r} = \textbf{ r}A_1$.  The result is then obtained from \cite{Li5}, Theorem 3.6.
\end{proof}

\section{Koszulity for augmented skew PBW extensions}

For a $\mathbb{K}$-vector space $V$, we use $T(V)$ to
denote the tensor algebra of $V$ over $\mathbb{K}$. If $R$ is a semisimple $\mathbb{K}$-algebra and $V$ is
an $R$-bimodule, then $T_R(V)$ will be the tensor algebra of $V$ over $R$.

\begin{definition}[\cite{Phan1}]\label{def.augmented} Let $R$ be a $\mathbb{K}$-algebra.
 \begin{enumerate}
\item[\rm(i)] A $\mathbb{K}$-algebra $A$   is said to be \emph{augmented} if  there is an ideal $A_+ \subset A$ such that $A= A_+ \bigoplus \mathbb{K} \cdot 1$.
 The augmentation is then  $\varepsilon : A \twoheadrightarrow \mathbb{K}$.
\item[\rm(ii)]  An  $R$-algebra $A$ is $R$-\emph{augmented} if  there is an ideal $A_+ \subset A$ such that $A= A_+ \bigoplus R \cdot 1$ as $R$-module (\cite{Phan2}, Definition II.2.1).
\item[\rm(iii)] For an $R$-augmented algebra $A$, a grading $A = \bigoplus_{n\geq 0}A_n$ is
\emph{compatible} with the augmentation if $A = T_R(V)/I$ for some finite-dimensional $R$-bimodule $V$ and finitely-generated homogeneous ideal $I \subset \sum_{i\geq 2} V^{\otimes i}$,  $A_n = V^{\otimes n}$ mod $I$, and $A_+ = \bigoplus_{n>0} A_n$. More specifically, we will say that A is
\emph{connected-graded} if $R = \mathbb{K}$ (\cite{Phan2}, Definition II.2.2).
\end{enumerate}
\end{definition}
If $A$ is $R$-augmented, we regard ${_AR}$ as the $A$-module $A/A_+$.

\begin{proposition}\label{prop.augm skew pbw}
A skew PBW extension  is an $R$-augmented algebra if and only if $A$ is pre-commutative and constant.
\end{proposition}
\begin{proof}
$\Rightarrow)$ Let  $A= \sigma(R)\langle x_1,\dots x_n\rangle = A_+ \bigoplus R \cdot 1= A_+ \bigoplus R$ an $R$-augmented skew PBW extension and let $r\in R\ \backslash \ \{0\}$. Then $r(x_ix_j)=  x_i(r x_j)= (x_ir)x_j=(\sigma_i(r)x_i+\delta_i(r))x_j=(\sigma_i(r)x_i)x_j+\delta_i(r)x_j=\sigma_i(r)(x_ix_j)+\delta_i(r)x_j$. As ${\rm Mon}(A)$ is $R$-base, then $\sigma_i(r)=r$ and $\delta_i(r)=0$, i.e., $A$ is constant. Now, as $x_jx_i$, $c_{i,j}x_ix_j\in A_+$, where $c_{i,j}$ is as in Definition \ref{def.skewpbwextensions}. Then $x_jx_i - c_{i,j}x_ix_j = r_0 + r_1x_1+\cdots + r_nx_n\in A_+$ with $r_0, r_1, \dots, r_n\in R$. As $r_1x_1+ \cdots + r_nx_n\in A_+$, $r_0\in A_+$ and so $r_0=0$, i.e., $A$ is pre-commutative.\\
$\Leftarrow)$ As $x_ir=rx_i$ and $x_jx_i - c_{i,j}x_ix_j \in Rx_1 +\cdots + Rx_n$, for each $r\in R$ and each $1\leq i\leq n$, then $A$ is an $R$-algebra with $A= A_+ \bigoplus R$ as $R$-module, where $A_+$ is the ideal generated by $\{x_1, \dots, x_n\}$.
\end{proof}

\begin{example}\label{ex.augmented skew} Let $A$ be skew PBW extension of a connected $\mathbb{K}$-algebra $R$.
 If $A$ is  pre-commutative and of endomorphism type, and   $f: A\to R$ is as in Proposition \ref{prop.homo}, then $f$ is an  augmentation and $A$ is an augmented algebra.
\end{example}

\begin{proposition}\label{prop.compatible}
If $A$ is an $R$-augmented and quasi-commutative skew PBW extension, then the graduation $Gr(A)$ from Proposition \ref{teo.Gr(A)} is compatible with the augmentation of $A$.
\end{proposition}

\begin{proof}
Let $A$ be an $R$-augmented skew PBW extension then by Proposition \ref{prop.augm skew pbw}, $A$ is pre-commutative and constant. Additionally, since $A$ is quasi-commutative then in Definition \ref{def.skewpbwextensions}, the condition {\rm (iii)} is  replaced by $x_ir=rx_i$, and condition {\rm (iv)} is replaced by $x_jx_i-c_{i,j}x_ix_j= 0$,  for  $c_{i,j}\in R \ \backslash \ \{0\}$ and $1\leq i,j\leq n$; i.e.,  $A = R\langle x_1,\dots x_n\rangle/\langle x_jx_i-c_{i,j}x_ix_j, 1\leq i,j\leq n, c_{i,j}\in R \ \backslash \ \{0\} \rangle$, is a graded algebra with
$A = \bigoplus_{p\geq 0} A_p$, where
\begin{equation}\label{eq.grad mod skew}
A_p:=\begin{cases} R & {\rm if}\ \ p=0\\ \{f\in A\mid  f \text{  is homogeneous of degree } p\} & {\rm if}\ \ p\ge 1.
\end{cases}
\end{equation}
 By Proposition \ref{remarkAsigma} and Theorem \ref{teo.Gr(A)}, there exists a quasicommutative
skew PBW extension $A^{\sigma}\cong Gr(A) \cong A$ of $R$ in $n$ variables $z_1,\dotsc, z_n$ defined by the relations  $z_ir= rz_i$, $z_jz_i=c_{i,j}z_iz_j$, for $1\le i\le n$. For $V$ the $R$-bimodule generated by $\{z_1,\dots z_n\}$ and $I=\langle z_jz_i-c_{i,j}z_iz_j \rangle$,  $T_R(V)\cong R\langle z_1,\dotsc, z_n\rangle  \cong R\langle x_1,\dotsc, x_n\rangle$, then $A= T_R(V)/I$.  Note that $I\subset  \sum_{s\geq 2} V^{\otimes s}$,  $A_p =  V^{\otimes p}$ mod $I$, and $A_+ = \bigoplus_{p>0} A_p$.
\end{proof}

\begin{remark}
Note that if $A$ is an $R$-augmented ($A$ is constant) and quasi-commutative skew PBW extension, then $A$ is semi-commutative.
\end{remark}

Let  $A$ be a connected algebra, finitely generated in degree 1. Let $E(A) = \bigoplus E^{n,m}(A)$ $= 
\bigoplus E^{n,m}_A (\mathbb{K}, \mathbb{K})$ be the associated bigraded Yoneda algebra of $A$ (where $n$ is the cohomology degree and $-m$ is the internal degree inherited from the grading on $A$). For an $R$-augmented algebra $A$, the algebra $E(A) = Ext_A(R, R)$
is called the Yoneda algebra. Set $E^n(A) = \bigoplus_m E^{n,m}(A)$. Then $A$ is said to be Koszul if
it satisfies either of the following (equivalent) definitions:
\begin{enumerate}
\item[\rm(i)] $E^{n,m}(A) = 0$ unless $n = m$.
\item[\rm(ii)] $E(A)$ is generated as an algebra by $E^{1,1}(A)$.
\item[\rm(iii)] The module $\mathbb{K}$ admits a \emph{linear free resolution}, i.e., a resolution by free $A$-modules
\[
\cdots \to P^2\to P^1\to P^0\to \mathbb{K}\to 0,
\]
such that $P_i$ is generated in degree $i$.
\end{enumerate}
Cassidy and Shelton in \cite{Cassidy2008} gave other generalization of Koszul algebras which are called $\mathcal{K}_2$ algebras.  The connected graded algebra $A$ is said to be $\mathcal{K}_2$ if $E(A)$ is generated as an algebra by $E^1(A)$ and $E^2(A)$. They shown that the algebra $A$ is Koszul if and only if it is quadratic homogeneous and  $\mathcal{K}_2$ (see Corollary 4.6). Note that if $A$ is a skew PBW extension and $A$ is  a graded algebra, then $A$ is homogeneous quadratic. Therefore, $\mathcal{K}_2$ algebras and Koszul algebras are the same for  skew PBW extensions.  Phan in \cite{Phan1} generalizes the definition of $\mathcal{K}_2$ algebras and Koszul algebras for augmented and $R$-augmented  algebras.
\begin{definition}\label{def.augm Koszul}
An augmented $\mathbb{K}$-algebra $A$ is $\mathcal{K}_2$ if $E^1(A)$ and $E^2(A)$ generate $E(A)$ as a $\mathbb{K}$-algebra. The $R$-augmented algebra $A$ is $\mathcal{K}_m$  if $E(A)$ is generated as an $R$-algebra by $E^1(A),\dots, E^m(A)$. If $A$ has  a grading compatible with its augmentation and $A$ is
$\mathcal{K}_1$, then $A$ is called $R$-\emph{augmented Koszul}. 
\end{definition}

\begin{theorem}[\cite{Phan2}, Theorem II.3.4]\label{teo.Koszul iff resol} $A$ is $R$-augmented Koszul  if and only if there exists a projective resolution $P^{\bullet}$ of graded $A$-modules for $R$ as  $A$-module, such that $P^n$ is generated in degree $n$.
\end{theorem}

Let $A=\sigma(R)\langle x_1,\dots, x_n\rangle$ be a skew PBW extension. We recall that each element $a\in A\ \backslash\ \{0\}$ has a unique representation as $a= c_0 + c_1X_1+\cdots+c_tX_t$, with $c_i\in R\ \backslash\ \{0\}$ and $X_i\in {\rm Mon}(A)\ \backslash\ \{1\}$ for $1\leq i\leq t$ and $c_0 \in R$ (\cite{LezamaGallego}, Remark 2).

\begin{proposition}\label{prop.homo}
Let $A=\sigma(R)\langle x_1,\dots, x_n\rangle$ be a skew PBW extension and let $f:  A\to R$ with $f(a):=c_0$, where $a$ is as above. Then $f$ is a ring homomorphism if and only if $A$ is pre-commutative and of endomorphism type.
\end{proposition}
\begin{proof}
$\Leftarrow)$:  $f$ is well defined since the representation of $a$ is unique. Is clear that $f(1)=1$, and $f(a+b)= f(a)+f(b)$ for all $a,b\in A$. Let $a= a_0 + a_1X_1+\cdots+a_tX_t$ and $b= b_0 + b_1Y_1+\cdots + b_sY_s$ the unique representation of $a, b\in A$. As $A$ is of derivation type and pre-commutative then
\begin{multline*}
ab=(a_0 + a_1X_1+\cdots +a_tX_t)(b_0 + b_1Y_1+\cdots + b_sY_s)\\
= a_0b_0+a_0b_1Y_1+\cdots a_0b_sY_s+\cdots + a_1X_1 b_0 + a_1X_1b_1Y_1\\
+\cdots + a_1X_1b_sY_s+\cdots + a_tX_tb_0+\cdots + a_tX_tb_sY_s\\
= a_0b_0+a_0b_1Y_1+\cdots a_0b_sY_s+\cdots + a_tX_tb_0+\cdots + c_{1,0}X_1 + c_{1,1}X_1Y_1\\
+\cdots + c_{1,s}X_1Y_s+\cdots + c_{t,0}X_t+\cdots + c_{t,s}X_tY_s\\=a_0b_0 + d_1Z_1+\cdots + d_nZ_n \text{ for } c_{i,j}\in R, d_k \in R\ \backslash\ \{0\} \text{ and } Z_i\in {\rm Mon}(A)\ \backslash\ \{1\}.
\end{multline*}
Then $f(ab)=a_0b_0=f(a)f(b)$.\\
$\Rightarrow)$: Let $r\in R\ \backslash\ \{0\}$ then for each $1\leq i\leq n$, $x_ir=\sigma_i(r)x_i+\delta_i(r)$. So, $f(x_ir)=f(\sigma_i(r)+\delta_i(r))=\delta_i(r)$. On the other hand $f(x_ir)=f(x_i)f(r)=0$. Then $\delta_i(r)=0$, and therefore $A$ is of endomorphism type. Now, $f(x_jx_i)=f(c_{i,j}x_ix_j+c_0+c_1x_1+\cdots + c_nx_n)=c_0= f(x_i)f(x_j)=0$, for each $1\leq i,j\leq n$ and $c_0, c_1,\dots, c_n \in R$. Then $A$ is pre-commutative.
\end{proof}
We have immediately  the following property for modules over skew PBW extensions.

\begin{corollary}\label{cor.sub estr}
If $A=\sigma(R)\langle x_1,\dots, x_n\rangle$ is a pre-commutative skew PBW extension of endomorphism type then  every $R$-module has the $A$-module structure.
\end{corollary}
\begin{proof}
Let $M$ be a $R$-module then for  $a= a_0 + a_1X_1+\cdots+a_tX_t\in A$  and $m\in M$, $a\cdot m:=a_0m$.
\end{proof}

\begin{remark} Let $f$ as in Proposition \ref{prop.homo}. In general $f$ is  $R$-homomorphism for all skew PBW extension $A$ of $R$. $f$ is surjective as $R$-homomorphism  and as ring homomorphism.
\end{remark}

In the next theorem, we study the Koszul property for $R$-augmented algebras, i.e., in the  sense of the Definition \ref{def.augm Koszul}.
\begin{theorem}
Let $R$ be a commutative augmented Koszul $\mathbb{K}$-algebra. If $A$ is a  semi-commutative skew PBW extension of $R$ then $A$ is $R$-augmented Koszul.
\end{theorem}

\begin{proof}
 Note that $R$ is an augmented $\mathbb{K}$-algebra, i.e., there is an ideal $R_+ \subset R$ such that $R= R_+ \bigoplus \mathbb{K} \cdot 1$. The augmentation is then $\varepsilon_R: R\to \mathbb{K}$; $R$ has a grading compatible with its augmentation, i.e., a grading $R= \bigoplus_{n\geq 0}R_n$ such that  $R = T(V)/I$ for some finite-dimensional $\mathbb{K}$-bimodule $V$ and finitely-generated homogeneous ideal $I \subset \sum_{i\geq 2} V^{\otimes i}$,  $R_n = V^{\otimes n}$ mod $I$, and $R_+ = \bigoplus_{n>0} R_n$. By Theorem \ref{teo.Koszul iff resol}, there exists a projective resolution
 \begin{equation}\label{resol.K}
\dotsb \xrightarrow{f_{i+1}} P^{i} \xrightarrow{f_i} \dotsb \xrightarrow{f_2} P^{1} \xrightarrow{f_1} P^{0} \xrightarrow{f_0} \mathbb{K} \to 0,
\end{equation}
of graded $R$-modules for ${_R\mathbb{K}}$ such that $P^n$ is generated in degree $n$. Since $A$ is a pre-commutative and constant skew PBW extension of $R$, then
 $\sigma_i=i_R$ and $\delta_i=0$ where $\sigma_i$ and $\delta_i$  are as in Proposition \ref{sigmadefinition}. So $A$ is of endomorphism type and of derivation type. By Proposition \ref{prop.augm skew pbw} $A$ is an $R$-augmented algebra, i.e., $A=A_+ \bigoplus R $ for some ideal $A_+$ of $A$.  By Corollary \ref{cor.sub estr} $R$ and $\mathbb{K}$ have the left and right $A$-module structure. Since every semi-commutative skew PBW extension is quasi-commutative, then by Proposition \ref{prop.compatible} the graduation $Gr(A)$ from Proposition \ref{teo.Gr(A)} is compatible with the augmentation of $A\cong A\otimes_R R$.  Since $R$ is commutative and $P_i$ is a left $R$-module then by  Corollary \ref{cor.sub estr} $P^i$ is a left and right $A$-module. So, $P^i\otimes_RA$ is a left projective $A$-module. Since $\mathbb{K}\cong R/R_+$ then by \cite{Rotman2009}, Proposition 2.68,  we have that $\mathbb{K}\otimes_R A\cong R/R_+\otimes_R A\cong A/R_+A\cong (A_+ \bigoplus R )/A_+\cong R$. Note that
 \begin{equation}\label{resol.A}
\dotsb \xrightarrow{f_{i+1}\otimes\iota_A}  P^{i}\otimes_R A  \xrightarrow{f_{i}\otimes\iota_A} \dotsb \xrightarrow{f_{2}\otimes\iota_A} P^{1}\otimes_R A \xrightarrow{f_{1}\otimes\iota_A} P^{0}\otimes_R A \xrightarrow{f_{0}\otimes\iota_A} \mathbb{K}\otimes_R A \cong R\to 0
\end{equation}
is a graded projective resolution of $R$ as a graded  $A$-module, where $P^n\otimes_R A$ is generated in degree $n$ (see proof of Theorem \ref{result.princ}).
Then, by Theorem \ref{teo.Koszul iff resol} we have  that $A$ is $R$-augmented Koszul.
\end{proof}

\end{document}